\documentclass[12pt]{amsart}
\usepackage{amsmath,amsfonts,euscript,amscd,amsthm,amssymb,upref,graphics}
\usepackage{mathrsfs}
\usepackage[all]{xy}
\usepackage{color}

\theoremstyle{definition}
\newtheorem*{question}{Question}

\swapnumbers

\theoremstyle{plain}

\newtheorem{theorem}{Theorem}[section]

\newtheorem{lemma}[theorem]{Lemma}

\newtheorem{corollary}[theorem]{Corollary}

\newtheorem{bigbigremark}[theorem]{Remark}

\theoremstyle{definition}

\newtheorem{definition}[theorem]{Definition}

\newtheorem{parag}[theorem]{}
\newtheorem{example}[theorem]{Example}

\newtheorem{notation}[theorem]{Notation}

\newtheorem{notations}[theorem]{Notations}

\newtheorem{bigremarks}[theorem]{Remarks}

\theoremstyle{remark}

\newtheorem*{remark}{Remark}

{\begin{enumerate}\setlength{\itemsep}{#1}}{\end{enumerate}}

{\vspace{2mm}\noindent\ref{#1}. {\bf #2.} \it}{\vspace{2mm}}

\newcommand{\Aut}{	\operatorname{{\rm Aut}}}
\newcommand{\Spec}{	\operatorname{\text{\rm Spec}}}

\newcommand{\dom}{	\operatorname{\text{\rm dom}}}
\newcommand{\image}{	\operatorname{\text{\rm im}}}
\newcommand{\trdeg}{	\operatorname{\text{\rm trdeg}}}
\newcommand{\Frac}{	\operatorname{\text{\rm Frac}}}

\newcommand{\Der}{	\operatorname{\text{\rm Der}}}

\newcommand{\lnd}{	\operatorname{\text{\rm LND}}}

\newcommand{\ML}{	\operatorname{\text{\rm ML}}}
\newcommand{\FML}{	\operatorname{\text{\rm FML}}}

\newcommand{\PE}{	\operatorname{\text{\rm PE}}}

\newcommand{\setspec}[2]{\big\{\,#1\, \mid \,#2\, \big\}}

\newcommand{\Nat}{\ensuremath{\mathbb{N}}}
\newcommand{\Rat}{\ensuremath{\mathbb{Q}}}
\newcommand{\Comp}{\ensuremath{\mathbb{C}}}
\newcommand{\Reals}{\ensuremath{\mathbb{R}}}
\newcommand{\aff}{\ensuremath{\mathbb{A}}}

\newcommand{\bk}{{\ensuremath{\rm \bf k}}}
\newcommand{\ck}{{\ensuremath{\bar\bk}}}
\newcommand{\kk}[1]{\bk^{[#1]}}

\newcommand{\pgoth}{{\ensuremath{\mathfrak{p}}}}
\newcommand{\qgoth}{{\ensuremath{\mathfrak{q}}}}

\newcommand{\mgoth}{{\ensuremath{\mathfrak{m}}}}

\newcommand{\Eeul}{\EuScript{E}}

\newcommand{\Seul}{\EuScript{S}}

\newcommand{\Xeul}{\EuScript{X}}

\newcommand{\isom}{\cong}
\renewcommand{\epsilon}{\varepsilon}
\renewcommand{\phi}{\varphi}
\renewcommand{\emptyset}{\varnothing}

\newenvironment{enumerata}%
{\begin{enumerate}

}{\end{enumerate}}

\newcommand{\rien}[1]{}

\begin{document}
\renewcommand{\baselinestretch}{1.07}


\title{Rings with trivial FML-invariant}

\author{Daniel Daigle}

\address{Department of Mathematics and Statistics\\
	University of Ottawa\\
	Ottawa, Canada\ \ K1N 6N5}

\email{ddaigle@uottawa.ca}

\thanks{Research supported by a grant from NSERC Canada.}

\keywords{Locally nilpotent derivation, unirational variety, affine variety, affine space.}

{\renewcommand{\thefootnote}{}
\footnotetext{2010 \textit{Mathematics Subject Classification.}
Primary: 14R10. Secondary: 14R20, 14M20, 14R05.}}

\begin{abstract} 
Let $\bk$ be a field of characteristic zero and $B$ a commutative integral domain that is also a finitely generated $\bk$-algebra.
It is well known that if $\bk$ is algebraically closed and the ``Field Makar-Limanov'' invariant $\FML(B)$ is equal to $\bk$,
then $B$ is unirational over $\bk$.
This article shows that, when $\bk$ is not assumed to be algebraically closed,
the condition $\FML(B)=\bk$ implies that there exists a nonempty Zariski-open subset $U$ of $\Spec B$ with the following property:
for each prime ideal $\pgoth \in U$, the $\kappa(\pgoth)$-algebra $\kappa(\pgoth) \otimes_\bk B$ can be embedded in a polynomial ring
in $n$ variables over $\kappa(\pgoth)$, where $n=\dim B$ and $\kappa(\pgoth) = B_{\pgoth}/{\pgoth}B_{\pgoth}$.
\end{abstract}
\maketitle
  
\vfuzz=2pt


\newcommand{\Ascr}{\mathscr{A}}
\newcommand{\Dscr}{\mathscr{D}}
\newcommand{\Kscr}{\mathscr{K}}

\newcommand{\Q}[1]{\mbox{$\mathbf{S}({#1})$}}
\newcommand{\bQ}[1]{\mbox{$\mathbf{\bar S}({#1})$}}

\section*{Introduction}

In this article, the word \textit{ring} means commutative ring with a unity.
By a \textit{domain}, we mean a commutative integral domain.
If $A$ is a domain then $\Frac(A)$ is its field of fractions.
If $\bk$ is a field, then a \textit{$\bk$-domain} is a domain that is also a $\bk$-algebra;
by an \textit{affine $\bk$-domain} we mean a $\bk$-domain that is finitely generated as a $\bk$-algebra.

If $B$ is a ring, a derivation $D: B \to B$ is \textit{locally nilpotent} if for each $x\in B$ there exists $n \in \Nat$ such that $D^n(x)=0$.
The set of locally nilpotent derivations $D : B \to B$ is denoted $\lnd(B)$.
One defines
$$
\ML(B) = \bigcap_{\mbox{\scriptsize $D \in \lnd(B)$}} \ker D 
\qquad \text{and} \qquad
\FML(B) = \bigcap_{\mbox{\scriptsize $D \in \lnd(B)$}} \Frac(\ker D) ,
$$
where in the second case $B$ is assumed to be a domain and the intersection is taken in $\Frac B$.
If $\bk$ is a field of characteristic zero and $B$ is a $\bk$-domain then $\bk \subseteq \ML(B) \subseteq \FML(B)$,
and if $\FML(B)=\bk$ then we say that \textit{$B$ has trivial $\FML$-invariant.}

Let $\bk$ be an algebraically closed field of characteristic zero and $B$ an affine $\bk$-domain.
It remained an open question for some time whether the condition $\ML(B)=\bk$ implied that $B$ is rational over $\bk$
(one says that $B$ is \textit{rational} over $\bk$ if the field extension $\Frac(B)/\bk$ is purely transcendental).
However, Liendo gave examples in \cite{Liendo_Rationality2010} (and so did Popov in \cite{Popov_Russellfest})
showing that the implication is false.
Liendo then conjectured that the stronger condition $\FML(B)=\bk$ would imply that $B$ is rational or at least unirational over $\bk$
(one says that $B$ is \textit{unirational} over $\bk$ if there exists a purely transcendental field extension $F/\bk$ of finite transcendence degree
such that $\bk \subseteq \Frac(B) \subseteq F$).
Then the following result was proved:
\begin{quote}
\textbf{Unirationality Theorem.} \it 
Let $\bk$ be an algebraically closed field of characteristic zero and $B$ an affine $\bk$-domain satisfying $\FML(B)=\bk$.
Then $B$ is unirational over $\bk$.
\end{quote}
This statement follows from either one of \cite[Prop.\ 5.1]{Arz_Flen_Kaliman_Kutz_Zaid:FlexAut} or \cite[Thm 4]{Popov_InfiniteDimAlgGps2014}.
Moreover, examples are given in  \cite{Popov_RatFML2013} showing that $B$, in the above statement, is not necessarily rational over $\bk$.

This article investigates what becomes of the Unirationality Theorem when $\bk$ is not assumed to be algebraically closed.
It is certainly the case that the condition $\FML(B)=\bk$ implies that $B$ is {\it geometrically unirational,}
i.e.,  that $\ck \otimes_\bk B$ is unirational over $\ck$, where $\ck$ denotes the algebraic closure of $\bk$
(this follows from the Unirationality Theorem and some straightforward technique, see Cor.\ \ref{9132g8rf7rhf}).
The aim of this article is to show that  $\FML(B)=\bk$ implies that $B$ satisfies a condition stronger than geometric unirationality.
Before describing this result, let us make a few remarks about Sections \ref{Section:embeddingsinpolynomialalgebras} and \ref{Secrion:TheposetsAscrBandKscrB}.

Given a field $\bk$ and an affine $\bk$-domain $B$, 
let $\Xeul_\bk(B)$  be the set of prime ideals $\pgoth$ of $B$ such that the $\kappa(\pgoth)$-algebra
$\kappa(\pgoth) \otimes_\bk B$ can be embedded in a polynomial ring in finitely many variables over $\kappa(\pgoth)$,
where we write $\kappa(\pgoth) = B_\pgoth / \pgoth B_\pgoth$ for each $\pgoth \in \Spec B$.
It is interesting to consider the class of affine $\bk$-domains $B$ satisfying the condition that $\Xeul_\bk(B)$ has nonempty interior
(i.e., some nonempty open subset of $\Spec B$ is included in $\Xeul_\bk(B)$).
Ex.\ \ref{pc09vnE230ed9vCf} shows that it is possible for $\Xeul_\bk(B)$ to be dense in $\Spec B$ and to have empty interior,
so the condition ``$\Xeul_\bk(B)$ has nonempty interior'' is strictly stronger than $\Xeul_\bk(B)$ being dense in $\Spec B$.
Although this implies that $\Xeul_\bk(B)$ is not always a constructible subset of $\Spec B$,
the main result of Section \ref{Section:embeddingsinpolynomialalgebras} (Thm \ref{cp0Q9vn23we9dfcwend0}) 
asserts that certain sets closely related to $\Xeul_\bk(B)$ are constructible.
Cor.\ \ref{pc0vW2n3ef0qK2Jfij0x03rh} asserts that an affine $\bk$-domain $B$ satisfies the condition ``$\Xeul_\bk(B)$ has nonempty interior''
if and only if the $\Frac(B)$-algebra $\Frac(B) \otimes_\bk B$ can be embedded in a polynomial ring $(\Frac B)[X_1,\dots,X_n]$ for some $n$.

Section \ref{Secrion:TheposetsAscrBandKscrB}
recalls (from \cite{Daigle:StructureRings}) some properties of the invariant $\Kscr(B)$ of the ring $B$. 
These facts are needed in Section \ref{RingshavingtrivialFMLinvariant}.

The main result of this paper (Thm \ref{9i3oerXfvdf93p04efeJ}) states that
{\it if $\bk$ is a field of characteristic zero and $B$ is an affine $\bk$-domain
satisfying  $\FML(B)=\bk$ then $\Xeul_\bk(B)$ has nonempty interior.}
Our proof makes use of Thm \ref{cp0Q9vn23we9dfcwend0} and of some results from \cite{Popov_InfiniteDimAlgGps2014}.
In the special case where $\bk$ is algebraically closed,
our result states that {\it if  $\FML(B)=\bk$ then $B$ can be embedded in a polynomial ring over $\bk$},
which is stronger than the statement that $B$ is unirational over $\bk$.

\medskip
To the notations and conventions already introduced in the above text,
we add the following.
We write $\subseteq$ for inclusion, $\subset$ for strict inclusion, and $\setminus$ for set difference.
We adopt the convention that $0 \in \Nat$.
If $A$ is a ring and $n \in \Nat$, $A^{[n]}$ denotes a polynomial ring in $n$ variables over $A$;
if $\bk$ is a field, $\bk^{(n)}$ denotes the field of fractions of $\kk n$.
We write $\trdeg_K(L)$ or $\trdeg(L:K)$ for the transcendence degree of a field extension $L/K$.
If $A \subseteq B$ are domains, the transcendence degree of $B$ over $A$ is defined to be that of $\Frac B$ over $\Frac A$,
and is denoted $\trdeg_A(B)$ or $\trdeg(B:A)$.
If $A$ is a ring then $A^*$ is its group of units, $\dim A$ is the Krull dimension of $A$ and
if $a \in A$ then $A_a = S^{-1}A$ where $S = \{1,a,a^2,\dots\}$.


\section{Embeddings in polynomial algebras}
\label{Section:embeddingsinpolynomialalgebras}

Throughout this section, $\bk$ is an arbitrary field.

Given a $\bk$-algebra $B$ and $n \in \Nat$,  we write $B \subseteq \kk n$ as an abbreviation for the sentence:
{\it there exists an injective homomorphism of $\bk$-algebras from $B$ to a polynomial algebra in $n$ variables over $\bk$.}

It follows that if $B$ is a $\bk$-algebra, $K$ an extension field of $\bk$ and $n \in \Nat$, 
the notation $K \otimes_\bk B \subseteq K^{[n]}$ means:
{\it  there exists an injective homomorphism of $K$-algebras from $K \otimes_\bk B$ to a polynomial algebra in $n$ variables over $K$.}

\begin{notation}
Given an algebra $B$ over a field $\bk$,
we write $\PE_\bk(B)$ for the class of all field extensions $K/\bk$ satisfying
$$
\text{$K \otimes_\bk B \subseteq K^{[n]}$ for some $n \in \Nat$.}
$$
\end{notation}

The notation ``PE'' stands for ``polynomial embedding'' (in the sense of ``embedding in a polynomial ring'').
Although $\PE_\bk(B)$ is not necessarily a set,\footnote{For instance if $B \subseteq \kk n$ then $\PE_\bk(B)$ is the class of \textit{all} field
extensions $K/\bk$, which is not a set in the sense of the ZFC axiomatization.}
 there is no harm in using set notations such as
``$K/\bk \in \PE_\bk(B)$'' or ``$\PE_\bk(B) \neq \emptyset$'' (whose meanings are obvious).
It follows from part (c) of the following fact 
that if $B$ is finitely generated then $\PE_\bk(B)$ is the class of extensions $K/\bk$ satisfying $K \otimes_\bk B \subseteq K^{[ \dim B ]}$.

\begin{lemma} \label {0si1heazd93g9w8las0}
Let $\bk$ be a field, $B$ a finitely generated $\bk$-algebra and $K$ an extension field of $\bk$.
\begin{enumerata}

\item $K \otimes_\bk B$ is a finitely generated $K$-algebra and $\dim( K \otimes_\bk B ) = \dim B$.

\item If $n \in \Nat$ is such that $K \otimes_\bk B \subseteq K^{[n]}$, then $n \ge \dim B$.

\item If there exists $n \in \Nat$ such that $K \otimes_\bk B \subseteq K^{[n]}$, then  $K \otimes_\bk B \subseteq K^{[ \dim B ]}$.

\end{enumerata}
\end{lemma}

\begin{proof}
It is clear that $K \otimes_\bk B$ is finitely generated.
Let $d = \dim B$. By Noether's Normalization Lemma there exists an injective $\bk$-homomorphism $\kk d \to B$ which is also integral.
Applying the functor $K \otimes_\bk(\underline{\ \ })$ gives  an injective and integral $K$-homomorphism from $K \otimes_\bk \kk d = K^{[d]}$
to $K \otimes_\bk B$, so  $\dim( K \otimes_\bk B ) = d$.
This proves (a). Assertions (b) and (c) follow from  Lemma B of \cite{Eakin72}.
\end{proof}

\begin{lemma} \label {XJ03rfwIe0d1Z20wdskfUw9e8X}
Let $B$ be an algebra over a field $\bk$ and suppose that $K/\bk \in \PE_\bk(B)$.
\begin{enumerata}

\item Every overfield $L$ of $K$ satisfies $L/\bk \in \PE_\bk(B)$.

\item If $B$ is finitely generated as a $\bk$-algebra then
there exists a finitely generated field extension $K_0/\bk$ such that $\bk \subseteq K_0 \subseteq K$ and $K_0/\bk \in \PE_\bk(B)$.

\end{enumerata}
\end{lemma}

\begin{proof}
(a) For some $n \in \Nat$, there exists an injective $K$-homomorphism $K\otimes_\bk B \to K^{[ n ]}$.
Applying $L \otimes_K (\underline{\ \ })$ gives an injective $L$-homomorphism from 
$L \otimes_K (K\otimes_\bk B) = L \otimes_\bk B$ to  $L \otimes_K K^{[ n ]} = L^{[ n ]}$.

(b) For some $n \in \Nat$, there exists an injective $K$-homomorphism $\phi : K \otimes_\bk B \to K[X_1,\dots,X_n] = K^{[n]}$.
Choose $b_1, \dots, b_s \in B$ such that $B = \bk[b_1, \dots, b_s]$.
There exists a finite subset $S$ of $K$ that contains all coefficients of the polynomials
$\phi( 1 \otimes b_i ) \in K[X_1,\dots,X_n]$, $1 \le i \le s$. 
Define $K_0 = \bk(S)$, then the image of the composite
$K_0 \otimes_\bk B  \to K \otimes_\bk B \xrightarrow{\phi} K[X_1,\dots,X_n]$ is included in  $K_0[X_1,\dots,X_n]$,
so $K_0 \otimes_\bk B  \subseteq K_0^{[n]}$ and hence $K_0/\bk \in \PE_\bk(B)$.
\end{proof}

\begin{lemma} \label {p0cwosHyoFto693ud}
Let $B$ be an algebra over a field $\bk$.
If $\PE_\bk(B) \neq \emptyset$ then $B$ is geometrically integral, i.e., $K \otimes_\bk B$ is a domain for every extension field $K$ of $\bk$.
\end{lemma}

\begin{proof}
Let $K$ be an extension field of $\bk$.
Choose an element $L/\bk$ of $\PE_\bk(B)$, and choose an algebraically closed field $M$ satisfying $L \subseteq M$ and $\trdeg_\bk(M) \ge \trdeg_\bk(K)$.
Then there exists a $\bk$-homomorphism $K \to M$, so the fact that $B$ is a flat $\bk$-module implies that  $K \otimes_\bk B$ is a subring of $M \otimes_\bk B$.
We have $M/\bk \in \PE_\bk(B)$ by Lemma \ref{XJ03rfwIe0d1Z20wdskfUw9e8X}, so $M \otimes_\bk B \subseteq M^{[ n] }$ for some $n$,
so $M \otimes_\bk B$ is a domain and hence $K \otimes_\bk B$ is a domain.
\end{proof}

\begin{lemma} \label {0cv9n2w0dZw0dI28efydldc0}
Consider a tensor product of rings
$$
\xymatrix@R=12pt{
S \ar[r] &   S \otimes_R T \\
R \ar[u] \ar[r]  &   T \ar[u]
}
$$
where all homomorphisms are injective.
\begin{enumerata}

\item Suppose that $S$ is a free $R$-module and that there exists a basis $\Eeul$ of $S$ over $R$ such that $1 \in \Eeul$.
Then $S \cap T = R$.

\item If $R,S,T$ and $S \otimes_R T$ are domains, and if
$(s_j)_{j \in J}$ is a family of elements of $S$ which is a transcendence basis of $\Frac S$ over $\Frac R$,
then $(s_j \otimes 1)_{j \in J}$ is a transcendence basis of $\Frac(S \otimes_R T)$ over $\Frac T$.
In particular, $\trdeg_T(S \otimes_R T) = \trdeg_R S$.

\end{enumerata}
\end{lemma}

\begin{proof}
Exercise left to the reader.
\end{proof}

\begin{notations} \label {90q3985yTghvj29hvnr}
Let $\bk$ be a field, let $R$ and $B$ be $\bk$-algebras,
let $N \in \Nat$ and let $R[X] = R[X_1, \dots, X_N] = R^{[N]}$.
Let $\Psi: B \to R[X]$ be a $\bk$-homomorphism. Then for each prime ideal $\pgoth \in \Spec R$ we define the following notations:
\begin{itemize}

\item $\phi_\pgoth : R \to \kappa(\pgoth)$ is the canonical homomorphism, where $\kappa(\pgoth) = R_\pgoth / \pgoth R_\pgoth$.

\item $\tilde\phi_\pgoth : R[X] \to \kappa(\pgoth)[X]$ is the induced homomorphism satisfying $\tilde\phi_\pgoth(X_i) = X_i$ for all $i$.

\item $\Psi^\pgoth : B \to \kappa(\pgoth)[X]$ is the 
composition  $B \xrightarrow{\Psi} R[X] \xrightarrow{\tilde\phi_\pgoth} \kappa(\pgoth)[X]$.

\item $\hat\Psi^\pgoth : \kappa(\pgoth) \otimes_\bk B \to \kappa(\pgoth)[X]$ is 
given by the universal property of the pushout:
\begin{equation} \label {892bf9jHGTiguYTFUKHiu73}
\begin{minipage}{.5\textwidth}
\small 
$$
\raisebox{7mm}{\xymatrix@R=15pt{
&& \kappa(\pgoth)[X] \\
\kappa(\pgoth) \ar[r]^(.4){ } \ar@/^1.3pc/[urr] &  \kappa(\pgoth) \otimes_\bk B \ar @{.>} [ur]^(.35){ \exists !\, \hat\Psi^\pgoth } \\ 
\bk \ar[r]_{ } \ar[u]^{ }  &  B \ar[u]^{ \beta_\pgoth } \ar[uur]_-{ \Psi^\pgoth }
}}
$$
\end{minipage}
\end{equation}

\end{itemize}
Note that $\Psi^\pgoth$ is a $\bk$-homomorphism and that $\hat\Psi^\pgoth$ is a $\kappa(\pgoth)$-homomorphism.
We define: 
$$
\Xeul_\bk( \Psi ) = \setspec{ \pgoth \in \Spec R }{ \text{$\hat\Psi^\pgoth$ is injective} } .
$$
\end{notations}

\begin{theorem}  \label {cp0Q9vn23we9dfcwend0}
Let the setup be as in paragraph \ref{90q3985yTghvj29hvnr}. If $B$ and $R$ are affine $\bk$-domains then the following hold.
\begin{enumerata}

\item \label {9hd8HGuKHytYf7624i-cons}
$\Xeul_\bk( \Psi )$ is a constructible subset of $\Spec R$.

\item \label {9hd8HGuKHytYf7624i-iff}
If $B$ is geometrically integral and $\mgoth$ is a maximal ideal of $R$, then $\mgoth \in \Xeul_\bk(\Psi)$ if and only if $\Psi^\mgoth$ is injective.

\end{enumerata}
\end{theorem}

\begin{proof}
Let $n=\dim B$.
We first prove \eqref{9hd8HGuKHytYf7624i-iff}.
Let $\mgoth$ be a maximal ideal of $R$ and consider diagram \eqref{892bf9jHGTiguYTFUKHiu73} with $\pgoth = \mgoth$.
Since the canonical homomorphism $B \to \kappa(\mgoth) \otimes_\bk B$ is injective, it is clear that if 
$\mgoth \in  \Xeul_\bk(\Psi)$ then $\Psi^\mgoth$ is injective. Conversely, suppose that $\Psi^\mgoth$ is injective.
Then $\trdeg_\bk( \image \Psi^\mgoth ) = n$.
By \eqref{892bf9jHGTiguYTFUKHiu73}, we have $\image(\Psi^\mgoth) \subseteq \image(\hat\Psi^\mgoth)$,
so $\trdeg_\bk(\image \hat\Psi^\mgoth ) \ge n$. 
Since $\kappa(\mgoth)$ is an algebraic extension of $\bk$,
$\trdeg_{\kappa(\mgoth)}( \image \hat\Psi^\mgoth ) = \trdeg_{\bk}( \image \hat\Psi^\mgoth ) \ge n$.
As $B$ is geometrically integral, $\dom( \hat\Psi^\mgoth ) =  \kappa(\mgoth) \otimes_\bk B$ is an affine $\kappa(\mgoth)$-domain of dimension $n$
(by Lemma \ref{0si1heazd93g9w8las0});
it follows that $\trdeg_{\kappa(\mgoth)}( \image \hat\Psi^\mgoth ) \ge \trdeg_{\kappa(\mgoth)}( \dom \hat\Psi^\mgoth )$,
so $\hat\Psi^\mgoth$ is injective and $\mgoth \in  \Xeul_\bk(\Psi)$. So  \eqref{9hd8HGuKHytYf7624i-iff} is proved.

To prove \eqref{9hd8HGuKHytYf7624i-cons}, we may assume that $\Xeul_\bk(\Psi) \neq \emptyset$. 
Then $\kappa(\pgoth) \otimes_\bk B \subseteq \kappa(\pgoth)[X] = \kappa(\pgoth)^{[N]}$ for any $\pgoth \in  \Xeul_\bk(\Psi)$,
so $B$ is geometrically integral by Lemma \ref{p0cwosHyoFto693ud}; also,  Lemma \ref{0si1heazd93g9w8las0} implies that $N\ge \dim B = n$ and:
$$
\textit{for each $\pgoth \in \Spec R$, $\kappa(\pgoth) \otimes_\bk B$ is an affine $\kappa(\pgoth)$-domain of dimension $n$.}
$$
The first step in the proof of \eqref{9hd8HGuKHytYf7624i-cons} consists in proving: 
\begin{equation} \label {p0c9n2Z30edYej2X30f9f}
\textit{If $N=\dim B$ and $0 \in \Xeul_\bk( \Psi )$ then $\Xeul_\bk( \Psi )$ contains a nonempty open subset of $\Spec R$.}
\end{equation}
Here, $0$ stands for the zero ideal of $R$.
Note that the codomain of $\Psi$ is $R[X] = R[X_1,\dots,X_n] = R^{[n]}$ (because $N=n$)
and consider the homomorphism $\hat \Psi : R \otimes_\bk B \to R[X]$ given by $\hat\Psi(r \otimes b) = r\Psi(b)$.
Let  $K = \Frac R$.
Since $0 \in \Xeul_\bk( \Psi )$, ${\hat\Psi^0} : K \otimes_\bk B \to K[X]$ is an injective $K$-homomorphism.
For each $\pgoth \in \Spec R$, there is a commutative diagram
$$
\xymatrix{
\kappa(\pgoth)[X] &   R[X] \ar[r]^{\tilde\phi_0} \ar[l]_{\tilde\phi_\pgoth}      &   K[X]   \\
\kappa(\pgoth) \otimes_\bk B \ar[u]^{\hat\Psi^\pgoth}   &   R \otimes_\bk B \ar[r] \ar[l] \ar[u]^{\hat\Psi}   &   K \otimes_\bk B \ar[u]^{\hat\Psi^0}
}
$$
Let $j \in \{1,\dots,n\}$.
Since $K \otimes_\bk B$ is an affine $K$-domain of dimension $n$ and ${\hat\Psi^0}$ is an injective $K$-homomorphism,
$X_j$ is algebraic over ${\hat\Psi^0}( K \otimes_\bk B )$;
since ${\hat\Psi^0}( K \otimes_\bk B )$ is a localization of $\hat\Psi( R \otimes_\bk B )$,
it follows that $X_j$ is algebraic over $\hat\Psi( R \otimes_\bk B )$,
so there exists a nonzero polynomial in one variable $P_j(T) \in \hat\Psi( R \otimes_\bk B )[T]$ such that $P_j( X_j ) = 0$.
Let $c_j \in R \otimes_\bk B$ be such that $\hat\Psi(c_j) \in \hat\Psi( R \otimes_\bk B ) \setminus \{0\}$ is the leading coefficient of $P_j(T)$.
Since $\hat\Psi(c_j)$ is a nonzero polynomial in $R[X]$, the ideal $I_j$ of $R$ generated by the coefficients of $\hat\Psi(c_j)$ is nonzero,
and consequently $U_j = \Spec R \setminus V( I_j )$ is a nonempty open subset of $\Spec R$.
If $\pgoth \in U_j$ then $\hat\Psi(c_j) \notin \pgoth[X]$ and consequently $\tilde\phi_\pgoth( \hat\Psi(c_j) ) \neq 0$.
Since $\tilde\phi_\pgoth( \hat\Psi(c_j) )$ is a coefficient of the polynomial $P_j^{( \tilde\phi_\pgoth )}(T) \in  \hat\Psi^\pgoth( \kappa(\pgoth) \otimes_\bk B )[T]$,
we have  $P_j^{( \tilde\phi_\pgoth )}(T) \neq 0$.
Now $\kappa(\pgoth) \subseteq \image(\hat\Psi^\pgoth) \subseteq \kappa(\pgoth)[X]$,
$P_j^{( \tilde\phi_\pgoth )}(T)$ is a nonzero polynomial in $T$ with coefficients in the ring $\image(\hat\Psi^\pgoth)$,
and $P_j^{( \tilde\phi_\pgoth )}( X_j ) = \tilde\phi_\pgoth( P_j(X_j)) = 0$; so $X_j$ is algebraic over $\image( \hat\Psi^\pgoth )$.

Consider the nonempty open subset $U = \bigcap_{j=1}^n U_j$ of $\Spec R$. Let $\pgoth \in U$.
The preceding paragraph implies that $X_1, \dots, X_n$ are algebraic over  $\image( \hat\Psi^\pgoth )$,
so $\kappa(\pgoth)[X]$ is algebraic over $\image(\hat\Psi^\pgoth)$
and consequently $\trdeg_{\kappa(\pgoth)}( \image \hat\Psi^\pgoth ) = n$;
as $\dom( \hat\Psi^\pgoth ) = \kappa(\pgoth) \otimes_\bk B$ is an affine $\kappa(\pgoth)$-domain of dimension $n$ and $\hat\Psi^\pgoth$ is 
a $\kappa(\pgoth)$-homomorphism, it follows that $\hat\Psi^\pgoth$ is injective and hence that $\pgoth \in \Xeul_\bk( \Psi )$.
This shows that $U \subseteq \Xeul_\bk( \Psi )$, so \eqref{p0c9n2Z30edYej2X30f9f} is proved.
Next, we generalize \eqref{p0c9n2Z30edYej2X30f9f} slightly. Let us prove:
\begin{equation} \label {7vJG872gd3f7gvkhiutdhiX}
\textit{If $0 \in \Xeul_\bk( \Psi )$ then $\Xeul_\bk( \Psi )$ contains a nonempty open subset of $\Spec R$.}
\end{equation}
By assumption, $\hat\Psi^0 : \kappa(0) \otimes_\bk B \to \kappa(0)[X]$ is injective.
By Lemma B of \cite{Eakin72} applied to $ \kappa(0) \subseteq \kappa(0) \otimes_\bk B \subseteq \kappa(0)[X]$,
there exists a $\kappa(0)$-homomorphism $\rho : \kappa(0)[X_1, \dots, X_N] \to \kappa(0)[Y_1, \dots, Y_n] = \kappa(0)^{[n]}$
such that $\rho \circ \hat\Psi^0 :  \kappa(0) \otimes_\bk B \to \kappa(0)[Y]$ is injective.
Since $\Frac R = \kappa(0)$, there exists $r \in R \setminus \{0\}$ such that if we define $A = R[1/r]$ then
the image of the composite $B \to \kappa(0) \otimes_\bk B \xrightarrow{ \rho \circ \hat\Psi^0 } \kappa(0)[Y]$ is included in $A[Y]$.
So there is a $\bk$-homomorphism $\Psi_1 : B \to A[Y]$ such that the following diagram commutes:
$$
\xymatrix{
\kappa(0) \otimes_\bk B  \ar[r]^-{ \rho \circ \hat\Psi^0 } & \kappa(0)[Y_1,\dots,Y_n]  \\
B  \ar[u] \ar[r]_-{ \Psi_1 }  &  A[Y_1,\dots,Y_n] \ar[u] 
}
$$
Since $\hat\Psi_1^0$ coincides exactly with $\rho \circ \hat\Psi^0$, which is injective, we have $0 \in \Xeul_\bk( \Psi_1 )$,
so $\Psi_1$ satisfies the hypothesis of \eqref{p0c9n2Z30edYej2X30f9f}.
Thus $\Xeul_\bk( \Psi_1 )$ contains a nonempty open subset of $\Spec A$.
It is straightforward to check that the open immersion $\Spec A \to \Spec R$ maps $\Xeul_\bk( \Psi_1 )$ into  $\Xeul_\bk( \Psi )$,
so $\Xeul_\bk( \Psi )$ contains a nonempty open subset of $\Spec R$ and \eqref{7vJG872gd3f7gvkhiutdhiX} is proved.
From this, let us deduce that the following is true:
\begin{equation} \label {98BjfohgfYhi8923o2948yhc}
\begin{minipage}[t]{.9\textwidth}
\it If $\Xeul_\bk( \Psi )$ is dense in $\Spec R$, then  $\Xeul_\bk( \Psi )$ contains a nonempty open subset of $\Spec R$.
\end{minipage}
\end{equation}
So assume that $\Xeul_\bk( \Psi )$ is dense in $\Spec R$.
Let $(b_1, \dots, b_n)$ be a family of elements of $B$ that is a transcendence basis of $\Frac B$ over $\bk$ (recall that  $n = \dim B = \trdeg_\bk B$),
and consider the family $\big( \Psi(b_i) \big)_{i=1}^n$ in $R[X]$;
we claim that $\big( \Psi(b_i) \big)_{i=1}^n$ is algebraically independent over $R$. 
Indeed, consider $G \in R[T_1, \dots, T_n] = R^{[n]}$  such that $G( \Psi(b_1), \dots, \Psi(b_n) ) = 0$.
Let $\pgoth \in \Xeul_\bk(\Psi)$.
Then the polynomial $G^{( \phi_\pgoth )} \in \kappa(\pgoth)[T_1, \dots, T_n]$ satisfies
$G^{( \phi_\pgoth )}\big(  \hat\Psi^\pgoth (1 \otimes b_1) , \dots,  \hat\Psi^\pgoth (1 \otimes b_n)  \big)
= G^{( \phi_\pgoth )}\big( \Psi^\pgoth (b_1), \dots, \Psi^\pgoth (b_n) \big)
= G^{( \phi_\pgoth )}\big( \tilde\phi_\pgoth(\Psi( b_1 )), \dots, \tilde\phi_\pgoth(\Psi( b_1 )) \big)
= \tilde\phi_\pgoth \big( G(\Psi( b_1 ), \dots, \Psi( b_n ) ) \big)
= \tilde\phi_\pgoth ( 0 ) = 0$.
Since the family $(1 \otimes b_i)_{i=1}^n$ of elements of $\kappa(\pgoth) \otimes_\bk B$
is algebraically independent over $\kappa(\pgoth)$ (Lemma \ref{0cv9n2w0dZw0dI28efydldc0}),
and since $\hat\Psi^\pgoth$ is an injective $\kappa(\pgoth)$-homomorphism,
the family $\big( \hat\Psi^\pgoth (1 \otimes b_i) \big)_{i=1}^n$ of elements of $\kappa(\pgoth)[X]$
is algebraically independent over $\kappa(\pgoth)$; so 
$G^{( \phi_\pgoth )} = 0$ and hence $G \in \pgoth[T_1, \dots, T_n]$.
Since this is true for every $\pgoth \in \Xeul_\bk(\Psi)$, and since $\Xeul_\bk( \Psi )$ is dense in $\Spec R$, it follows that $G=0$.
This proves that $\big( \Psi(b_i) \big)_{i=1}^n$ is algebraically independent over $R$. 
Since $\phi_0 : R \to \kappa(0)$ is the inclusion of $R$ in its field of fractions,
the image  $\big( \tilde\phi_0(\Psi(b_i)) \big)_{i=1}^n$ of $\big( \Psi(b_i) \big)_{i=1}^n$ by $\tilde\phi_0 : R[X] \to \kappa(0)[X]$ 
is a family of elements of $\kappa(0)[X]$ that is algebraically independent over $\kappa(0)$. 
Since $\tilde\phi_0(\Psi(b_i)) = \Psi^0(b_i) = \hat\Psi^0( 1 \otimes b_i )$,
$\big( \tilde\phi_0(\Psi(b_i)) \big)_{i=1}^n$ is a family of elements of $\image\hat\Psi^0$,
so $\trdeg_{ \kappa(0) }( \image\hat\Psi^0 ) \ge n$.
By Lemma \ref{0si1heazd93g9w8las0} and the fact that $B$ is geometrically integral,
$\dom\hat\Psi^0 = \kappa(0) \otimes_\bk B$ is an affine $\kappa(0)$-domain of dimension equal to $\dim B = n$,
so  $\trdeg_{ \kappa(0) }( \image\hat\Psi^0 ) \ge \trdeg_{ \kappa(0) }( \dom\hat\Psi^0 )$ and hence
$\hat\Psi^0 : \kappa(0) \otimes_\bk B \to \kappa(0)[X]$ is injective.
Thus $0 \in \Xeul_\bk(\Psi)$, so \eqref{7vJG872gd3f7gvkhiutdhiX} implies that 
$\Xeul_\bk(\Psi)$ contains a nonempty open subset of $\Spec R$. So \eqref{98BjfohgfYhi8923o2948yhc} is proved.

For each $\pgoth \in \Spec R$, we write $V(\pgoth) = \setspec{ \qgoth \in \Spec R }{ \pgoth \subseteq \qgoth }$.
To show that $\Xeul_\bk( \Psi )$ is constructible, it suffices (by Prop.\ 6.C of \cite{Matsumura}) to prove the following:
\begin{equation} \label {p9cvn23efCvd0i}
\begin{minipage}[t]{.9\textwidth}
\it If $\pgoth \in \Spec R$ is such that $V(\pgoth) \cap \Xeul_\bk( \Psi )$ is dense in $V(\pgoth)$, 
then  $V(\pgoth) \cap \Xeul_\bk( \Psi )$ contains a nonempty open subset of $V(\pgoth)$. 
\end{minipage}
\end{equation}
To prove this, consider $\pgoth \in \Spec R$ such that $V(\pgoth) \cap \Xeul_\bk( \Psi )$ is dense in $V(\pgoth)$.
The canonical homomorphism $R \to R/\pgoth$ extends to a homomorphism $R[X] \to (R/\pgoth)[X]$ that sends each $X_i$ to itself;
let us define $\Psi_1 : B \to (R/\pgoth)[X]$ to be the composition $B \xrightarrow{\Psi} R[X] \to (R/\pgoth)[X]$.
Then it is straightforward to verify that the canonical homeomorphism $f : \Spec(R/\pgoth) \to V(\pgoth)$ satisfies
$f\big( \Xeul_\bk( \Psi_1 ) \big) = V(\pgoth) \cap \Xeul_\bk( \Psi )$.
Since $V(\pgoth) \cap \Xeul_\bk( \Psi )$ is dense in $V(\pgoth)$, it follows that $\Xeul_\bk( \Psi_1 )$ is dense in $\Spec(R/\pgoth)$.
By \eqref{98BjfohgfYhi8923o2948yhc}, we obtain that $\Xeul_\bk( \Psi_1 )$ contains a nonempty open subset of $\Spec(R/\pgoth)$.
So $f\big( \Xeul_\bk( \Psi_1 ) \big) = V(\pgoth) \cap \Xeul_\bk( \Psi )$ contains  a nonempty open subset of $V(\pgoth)$.
So \eqref{p9cvn23efCvd0i} is proved, and so is the Theorem.
\end{proof}

\begin{corollary} \label {pc9802oi3we0fdwUe}
Let $\bk$ be a field and $B$ an affine $\bk$-domain such that $\PE_\bk(B) \neq \emptyset$.
\begin{enumerata}

\item \label {9239r0rejg03er} There exists an affine $\bk$-domain $R$ satisfying $\Frac(R)/\bk \in \PE_\bk(B)$.

\item \label {90dsGHjFGugkjgUYT83} For any $R$ as in \eqref{9239r0rejg03er}, there exists a nonempty open subset $U$ of $\Spec R$ satisfying 
$$
\kappa(\pgoth) / \bk  \in  \PE_\bk(B) \ \ \text{for all $\pgoth \in U$,}
$$
where we write $\kappa(\pgoth) = R_\pgoth / \pgoth R_\pgoth$.

\end{enumerata}
\end{corollary}

\begin{proof}
Assertion \eqref{9239r0rejg03er} follows from Lemma \ref{XJ03rfwIe0d1Z20wdskfUw9e8X}.
We prove \eqref{90dsGHjFGugkjgUYT83}.
Let $K = \Frac R$. Since $K/\bk \in \PE_\bk(B)$, there exists an injective
$K$-homomorphism $\phi : K \otimes_\bk B \to K[X] =K[X_1,\dots,X_n] = K^{[n]}$ (for some $n \in \Nat$).
There exists $r \in R \setminus \{0\}$ such that, if we define $A = R[1/r]$, the image of the composite $B \to K \otimes_\bk B \xrightarrow{\phi} K[X]$
is included in $A[X]$. 
So we may consider the unique $\bk$-homomorphism $\Psi : B \to A[X]$ that makes
$$
\xymatrix@R=13pt{
K \otimes_\bk B \ar[r]^-{\phi}   &   K[X]  \\
B  \ar[u] \ar[r]_-{\Psi} &   A[X]  \ar[u]
}
$$
commute.
If $0$ denotes the zero ideal of $A$ then $\hat\Psi^0 = \phi$, so $\hat\Psi^0$ is injective and hence $0 \in \Xeul_\bk( \Psi )$; 
so $\Xeul_\bk( \Psi )$ is a dense subset of $\Spec A$.
Since $\Xeul_\bk( \Psi )$ is constructible by Thm \ref{cp0Q9vn23we9dfcwend0}, 
$\Xeul_\bk( \Psi )$ contains a nonempty open subset $U'$ of $\Spec A$.
Noting that $f : \Spec A \to \Spec R$ is an open immersion, we see that $U = f(U')$ is a nonempty open subset of $\Spec R$ that has the desired property.
\end{proof}

\begin{corollary} \label {kxjcop923wpeds}
Let $\bk$ be a field and $B$ an affine $\bk$-domain such that $\PE_\bk(B) \neq \emptyset$.
\begin{enumerata}

\item There exists a finite extension $\bk'$ of $\bk$ satisfying $\bk' / \bk \in \PE_\bk(B)$.

\item The algebraic closure $\ck$ of $\bk$ satisfies $\ck/\bk \in \PE_\bk(B)$.

\end{enumerata}
\end{corollary}

\begin{proof}
By Cor.\ \ref{pc9802oi3we0fdwUe},
there exists an affine $\bk$-domain $R$ such that $\Frac(R) / \bk \in \PE_\bk(B)$ and 
a nonempty open subset $U$ of $\Spec R$ such that $\kappa(\pgoth)/\bk \in \PE_\bk(B)$ for all $\pgoth \in U$,
where $\kappa(\pgoth) = R_\pgoth / \pgoth R_\pgoth$. Choose a maximal ideal $\mgoth$ of $R$ such that $\mgoth \in U$.
Then $\kappa(\mgoth) / \bk \in \PE_\bk(B)$ and $\kappa(\mgoth) / \bk$ is a finite extension, so (a) is proved.
As $\ck$ is an overfield of any finite extension $\bk'$ of $\bk$,
assertion (b) follows from (a) and Lemma \ref{XJ03rfwIe0d1Z20wdskfUw9e8X}. 
\end{proof}

\begin{definition}  \label {cvcn2o39wscnpwdni}
We say that a field extension $L/K$ \textit{has the density property} if it is finitely generated and the following equivalent conditions hold:
\begin{itemize}

\item for some affine $K$-domain $R$ satisfying $\Frac R = L$, $K$-rational points are dense in $\Spec R$;
\item for every affine $K$-domain $R$ satisfying $\Frac R = L$, $K$-rational points are dense in $\Spec R$.

\end{itemize}
\end{definition}

Note the following descent property for $\PE_\bk(B)$:

\begin{corollary}  \label {Ocj30edif2931402dfno}
Let $\bk \subseteq K \subseteq L$ be fields and let $B$ be an affine $\bk$-domain.
If $L/\bk \in \PE_\bk(B)$ and $L/K$ has the density property, then $K/\bk \in \PE_\bk(B)$.
\end{corollary}

\begin{proof}
The fact that $L/\bk \in \PE_\bk(B)$ implies that  (for some $n$) there
is an injective $L$-homomorphism from $L \otimes_\bk B = L \otimes_K (K \otimes_\bk B)$ to $L^{[n]}$,
so $L \otimes_K (K \otimes_\bk B) \subseteq L^{[n]}$ and consequently $L/K \in \PE_K( K \otimes_\bk B )$. 
Since $L/K$ has the density property, it is finitely generated; so there exists an affine $K$-domain $R$ satisfying $\Frac R = L$.
Then $\Frac(R) / K \in \PE_K( K \otimes_\bk B )$, 
so Cor.\ \ref{pc9802oi3we0fdwUe} implies that there exists a nonempty open subset $U$ of $\Spec R$ such that
$\kappa(\pgoth)/K \in \PE_K( K \otimes_\bk B)$ for all $\pgoth \in U$, where $\kappa(\pgoth) = R_\pgoth / \pgoth R_\pgoth$.
Since $L/K$ has the density property, there exists a $K$-rational point in $U$, i.e., there exists $\pgoth \in U$ such that $\kappa(\pgoth) = K$.
Thus $K/K \in  \PE_K( K \otimes_\bk B )$ and consequently $K \otimes_\bk B \subseteq K^{[n]}$.
This means that $K / \bk \in \PE_\bk(B)$, so we are done.
\end{proof}

\begin{notation} \label {p0c9f2n3w9sdc0qpwo}
Given an algebra $B$ over a field $\bk$, define 
{\small
\begin{align*}
\Xeul_\bk(B) &= \setspec{ \pgoth \in \Spec B }{ \kappa(\pgoth)/\bk \in \PE_\bk(B) }
= \setspec{ \pgoth \in \Spec B }{ \kappa(\pgoth) \otimes_\bk B \subseteq \kappa(\pgoth)^{[n]} \text{ for some $n \in \Nat$} }
\end{align*}
}
where $\kappa(\pgoth) = B_\pgoth / \pgoth B_\pgoth$.
Note that $\Xeul_\bk(B) = \setspec{ \pgoth \in \Spec B }{ \kappa(\pgoth) \otimes_\bk B \subseteq \kappa(\pgoth)^{[ \dim B ]} } $
when $B$ is finitely generated, by Lemma \ref{0si1heazd93g9w8las0}.
\end{notation}

One says that $\Xeul_\bk(B)$ \textit{has nonempty interior} if some nonempty open subset of $\Spec B$ is included in $\Xeul_\bk(B)$.
For reasons explained in the Introduction (and further explained in Section \ref{RingshavingtrivialFMLinvariant}),
we are interested in $\bk$-algebras $B$ such that $\Xeul_\bk(B)$ has nonempty interior.

\begin{corollary} \label {pc0vW2n3ef0qK2Jfij0x03rh}
Given a field $\bk$ and an affine $\bk$-domain $B$,
$$
\text{$\Xeul_\bk(B)$ has nonempty interior} \iff \Frac(B)/\bk \in \PE_\bk(B) .
$$
\end{corollary}

\begin{proof}
If $\Xeul_\bk(B)$ has nonempty interior then the generic point of $\Spec B$ is an element of $\Xeul_\bk(B)$, so $\Frac(B)/\bk \in \PE_\bk(B)$.
The converse is the case $R=B$ of Cor.\ \ref{pc9802oi3we0fdwUe}.
\end{proof}

\begin{bigbigremark} \label {d023kHhvjhhJHkru83ee}
Let $\bk$ be a field and $B$ an affine $\bk$-domain such that $\Xeul_\bk(B)$ has nonempty interior.
Then the following are equivalent:
\begin{enumerata}
\item  $B \subseteq \bk^{[ \dim B ]}$ 
\item  $\bk$-rational points are dense in $\Spec B$
\item  the extension $\Frac(B) / \bk$ has the density property
\item  $B$ is unirational over $\bk$.
\end{enumerata}
\end{bigbigremark}

\begin{proof}
Let $n = \dim B$.
By assumption, there exists a dense open subset $U$ of $\Spec B$ such that $U \subseteq \Xeul_\bk(B)$.
Then $\kappa(\pgoth) \otimes_\bk B \subseteq \kappa(\pgoth)^{[n]}$ for all $\pgoth \in U$.
If (b) holds, there exists a $\bk$-rational point $\pgoth \in U$; then $\kappa(\pgoth) = \bk$ and hence $B \subseteq \kk n$,
proving that (b)$\Rightarrow$(a).
Implications (a)$\Rightarrow$(d)$\Rightarrow$(b) are clear
and (c)$\Leftrightarrow$(b) is part of Def.\ \ref{cvcn2o39wscnpwdni}.
\end{proof}

The following example
shows that $\Xeul_\bk(B)$ is not always a constructible subset of $\Spec B$.

\begin{example} \label {pc09vnE230ed9vCf}
Suppose that $\bk$ is a field of characteristic $2$ and that $a \in \bk$ does not have a square root in $\bk$.
Define $B = \bk[X,Y] / (Y^2 + aX^2 +X)$, where $\bk[X,Y] = \kk2$. 
We leave it to the reader to check that for any extension field $K$ of $\bk$,
$$
K/\bk \in \PE_\bk(B) \iff K \otimes_\bk B = K^{[1]} \iff \text{$K$ contains a square root of $a$}.
$$
Consequently, 
$\Xeul_\bk(B) = \setspec{ \pgoth \in \Spec B }{ \text{$\kappa(\pgoth)$ contains a square root of $a$} }$.
Let us argue that both $\Xeul_\bk(B)$ and $\Spec(B) \setminus \Xeul_\bk(B)$ are infinite sets.

Given $\lambda \in \bk$, $( Y^2 + aX^2 +X, (X-\lambda)^2-a )$ is a proper ideal of $\bk[X,Y]$
since the equations  $Y^2 + aX^2 +X = 0$ and $(X-\lambda)^2=a$ have solutions in the algebraic closure of $\bk$.
It follows that there exists a maximal ideal $\mgoth_\lambda$ of $B$ such that $(x-\lambda)^2-a \in \mgoth_\lambda$,
where $x,y \in B$ are the canonical images of $X,Y$.
We have $\mgoth_\lambda \in \Xeul_\bk(B)$, and it is clear that $\lambda \mapsto \mgoth_\lambda$ is injective, so 
$\Xeul_\bk(B)$ is an infinite set ($\bk$ is an infinite field, since it is not perfect).

One can see directly that $\Frac(B) = \bk( y/x )$, so $B$ is rational over $\bk$ and hence $\Spec B$ has infinitely many $\bk$-rational points.
If $\pgoth \in \Spec B$ is a $\bk$-rational point then $\kappa(\pgoth)=\bk$ does not contain a square root of $a$, so $\pgoth \notin \Xeul_\bk(B)$.
So $\Spec(B) \setminus \Xeul_\bk(B)$ is an infinite set.
Note in particular that $\Xeul_\bk(B)$ is not a constructible subset of $\Spec(B)$.
\end{example}

One should note that $\Spec B$, in the above example, is a nontrivial form of the affine line over a field of positive characteristic.
In characteristic zero, it is not known whether $\Xeul_\bk(B)$ is always a constructible subset of $\Spec B$.\footnote{It can be shown
that if $\bk$ is a field of characteristic zero and $B$ is a $1$-dimensional affine $\bk$-domain
then $\Xeul_\bk(B)$ is either empty or equal to $\Spec B$, so in particular $\Xeul_\bk(B)$ is constructible.}
In the following characteristic zero example we are unable to decide whether $\Xeul_\bk(B)$ is constructible, but it seems plausible that it is not.

\begin{example} \label {cp9vnc2o3w9efc0IJ2w}
Let $\bk$ be the field of fractions of the domain $\Reals[u,v] / (u^2 + v^2 + 1)$ (where $\Reals[u,v] = \Reals^{[2]}$)
and let $B = \bk[X,Y,Z] / (X^2 + Y^2 + Z^2 + 1)$ (where $\bk[X,Y,Z] = \bk^{[3]}$). 
One can see that $\bk$ does not contain a square root of $-1$, so $\Reals$ is algebraically closed in $\bk$.
We claim:
\begin{enumerata}

\item $\bk$-rational points are dense in $\Spec B$;

\item $\Xeul_\bk(B)$ is dense in $\Spec B$;

\item $\Xeul_\bk(B)$ is a constructible subset of $\Spec B$ $\iff$ $B \subseteq \kk2$.

\end{enumerata}
To prove (a), observe that given any pair $(a,b) \in \Reals^2$ satisfying $a^2 + b^2 > 1$ the triple
$$
(x,y,z) =  \big( au+bv, bu-av, \sqrt{a^2+b^2-1} \big) \in \bk^3
$$
satisfies $x^2+y^2+z^2+1=0$ and hence determines a $\bk$-rational point of $\Spec B$.
It can be seen that this collection of points is dense in $\Spec B$, so (a) is true.

(b) Write $\kappa(\pgoth) = B_\pgoth / \pgoth B_\pgoth$ for each $\pgoth \in \Spec B$, and note that
\begin{equation} \label {xcW0dn120wdcdhfpqwc9}
\text{if $\kappa(\pgoth)$ contains a square root of $-1$ then $\pgoth \in \Xeul_\bk(B)$.}
\end{equation}
Indeed, if $i \in \kappa(\pgoth)$ satisfies $i^2=-1$ and if we define $X_1 = X+iY$ and $Y_1 = X-iY$
then $\kappa(\pgoth) \otimes_\bk B = \kappa(\pgoth)[X,Y,Z]/(X^2+Y^2+Z^2+1) = \kappa(\pgoth)[X_1,Y_1,Z]/(X_1Y_1+Z^2+1)$
is a ``Danielewski surface'' and hence $\kappa(\pgoth) \otimes_\bk B \subseteq \kappa(\pgoth)^{[2]}$, as is well known.
Then $\pgoth \in \Xeul_\bk(B)$ and \eqref{xcW0dn120wdcdhfpqwc9} is proved.
Consider $R = \Reals[X,Y,Z]/(X^2+Y^2+Z^2+1)$ and observe that $B = \bk \otimes_\Reals R$, 
so the canonical morphism $f : \Spec B \to \Spec R$ is an open map by \cite[Tag 037G]{stacks-project}.
Now let $U$ be a nonempty open subset of $\Spec B$;
then $f(U)$ is a nonempty open subset of $\Spec R$, so in particular there exists a maximal ideal $\mgoth$ of $R$ such that $\mgoth \in f(U)$.
So we may choose $\pgoth \in U$ such that $f(\pgoth) = \mgoth$. Since $R/\mgoth \isom \Comp$ and $\kappa(\pgoth)$ is an extension of $R/\mgoth$,
it follows that $\kappa(\pgoth)$ contains a square root of $-1$
and hence that $\pgoth \in \Xeul_\bk(B)$, by \eqref{xcW0dn120wdcdhfpqwc9}. This proves (b).

(c) If $\Xeul_\bk(B)$ is constructible then, by (b), $\Xeul_\bk(B)$ has nonempty interior;
by (a) and Rem.\ \ref{d023kHhvjhhJHkru83ee}, it follows that $B \subseteq \kk2$.
Conversely, if  $B \subseteq \kk2$ then $\Xeul_\bk(B) = \Spec B$, so $\Xeul_\bk(B)$ is constructible.

Thus (a), (b) and (c) are true.
Because $\bk$ does not contain a square root of $-1$, we cannot imagine how to embed $B$ in  $\kk2$;
in that sense, it seems plausible that $\Xeul_\bk(B)$ is not constructible.
However, we don't know if the condition  $B \subseteq \kk2$ is true, so we don't know if $\Xeul_\bk(B)$ is constructible.
\end{example}

\section{The posets $\Ascr(B)$ and $\Kscr(B)$}
\label {Secrion:TheposetsAscrBandKscrB}

Paragraph \ref{pc9293ed0wdjo03} states some basic facts about locally nilpotent derivations.
For background on this, we refer the reader to any of \cite{VDE:book}, \cite{Freud:Book} or \cite{Dai:IntroLNDs2010}.

\begin{parag} \label {pc9293ed0wdjo03}
Let $B$ be a domain of characteristic zero.
\begin{enumerata}

\item[(i)] Let $D \in \lnd(B)$ and write $A = \ker D$.
Then $A$ is factorially closed in $B$, i.e.,
the implication $xy \in A \Rightarrow x,y \in A$ is true for all $x,y \in B \setminus \{0\}$.
It follows that $A^* = B^*$ and hence that if $K$ is any field included in $B$ then $K\subseteq A$.
Moreover, if $B$ is a UFD then so is $A$.

\item[(ii)] Let $D \in \lnd(B) \setminus \{0\}$ and write $A = \ker D$.
Then there exists $s \in B$ satisfying $D(s) \neq 0$ and $D^2(s)=0$,
and given any such element $s$ we have $B_a = A_a[s] = A_a^{[1]}$, where we set $a = D(s)$.
It follows that if we define $K = \Frac A$ then $B \subseteq K[s] = K^{[1]}$ and $\Frac B = K^{(1)}$ (in particular $K$ is algebraically closed in $\Frac B$).
Moreover, if $0 \neq b \in B \subseteq K[s]$ then the degree of $b$ as a polynomial in $s$ is equal to the greatest
$n \in \Nat$ satisfying $D^n(b) \neq 0$.

\end{enumerata}
\end{parag}

We now recall the definition of the posets  $\Ascr(B)$ and $\Kscr(B)$, which are invariants of the ring $B$.
These objects are defined and studied in \cite{Daigle:StructureRings}.
In the present paper we are mostly interested in $\Kscr(B)$.

\begin{parag} \label {0cd20FXGICHGgt86cjF5i5rgo909}
Let $B$ be a domain of characteristic zero.
Given a subset $\Delta$ of $\lnd(B)$, define
$A_{\Delta} = \bigcap_{D \in \Delta} \ker D$ and 
$K_{\Delta} = \bigcap_{D \in \Delta} \Frac(\ker D)$,
where the first intersection is taken in $B$ and the second in $\Frac B$
(in particular,  $A_\emptyset = B$ and $K_\emptyset = \Frac B$).
Then define the two sets
$$
\Ascr(B) = \setspec { A_{\Delta} }{ \Delta \subseteq \lnd(B)  } 
\quad \text{and} \quad
\Kscr(B) = \setspec { K_{\Delta} }{ \Delta \subseteq \lnd(B)  } .
$$
Note that $\Ascr(B)$ is a nonempty set of subrings of $B$;
$(\Ascr(B), \subseteq)$ is a poset, its greatest element is $B$ and its least element is $\ML(B)$.
 Similarly, $\Kscr(B)$ is a nonempty set of subfields of $\Frac B$ whose
greatest element is $\Frac B$ and whose least element is $\FML(B)$.
\end{parag}

Refer to Section 3 of \cite{Daigle:StructureRings} for the proofs of Lemmas \ref{0ckj238hqn239hnfpawrnHzb}, \ref{f2i3p0f234hef}
and \ref{evdhd838eh65433ru0}.

\begin{lemma} \label {0ckj238hqn239hnfpawrnHzb}
If $B$ is a domain of characteristic zero then
each element of $\Ascr(B)$ is factorially closed in $B$ and 
each element of $\Kscr(B)$ is algebraically closed in $\Frac B$.
In particular, $\FML(B)$ is algebraically closed in $\Frac B$.
\end{lemma}

\begin{lemma} \label {f2i3p0f234hef}
Let $B$ be a domain of characteristic zero.
\begin{enumerata}

\item Let $D \in \lnd(B)$ and let $D' \in \Der( \Frac B )$ be the unique extension of $D$ to a derivation of $\Frac B$.
Then $\Frac(\ker D) = \ker D'$.

\item For any subset $\Delta$ of $\lnd(B)$, we have $B \cap K_\Delta = A_\Delta$.
Consequently,
$$
\Kscr(B) \to \Ascr(B),\ \ K \mapsto B \cap K,
$$
is a well-defined surjective map and moreover $B \cap \FML(B) = \ML(B)$.

\end{enumerata}
\end{lemma}

Recall that if $B$ is a $\Rat$-algebra then each $D \in \lnd(B)$ determines an automorphism 
$\exp(D)$ of the ring $B$, defined by $b \mapsto \sum_{i=0}^\infty \frac{D^n(b)}{n!}$ for $b \in B$.
In the case where $B$ is a domain, $\exp(D)$ has a unique extension to an automorphism of $\Frac B$.

\begin{lemma} \label {73npgjk32ncrg93rcn}
Let $\bk$ be a field of characteristic zero, $B$ a $\bk$-domain and $\Delta \subseteq \lnd(B)$.  Then
$$
K_\Delta = \setspec{ \xi \in \Frac B }{ \forall_{D \in \Delta}\ E_D(\xi) = \xi } 
= \setspec{ \xi \in \Frac B }{ \forall_{a \in \bk} \forall_{D \in \Delta}\ E_{a D}(\xi) = \xi } ,
$$
where for each $D \in \lnd(B)$ we let $E_D \in \Aut( \Frac B )$ be the unique extension of $\exp(D) \in \Aut(B)$. 
\end{lemma}

\begin{remark}
This can be written as $K_\Delta = (\Frac B)^G = (\Frac B)^{G'}$, where $G$ (resp.\ $G'$) denotes the subgroup of $\Aut( \Frac B )$ 
generated by $\setspec{ E_D }{ D \in \Delta }$ (resp.\ by $\setspec{ E_{aD} }{ \text{$a \in \bk$ and $D \in \Delta$} }$).
\end{remark}

\begin{proof}[Proof of Lemma \ref{73npgjk32ncrg93rcn}]
It suffices to prove that, for each $D \in \lnd(B)$,
$$
\Frac( \ker D ) =  \setspec{ \xi \in \Frac B }{ E_D(\xi) = \xi } .
$$
Moreover, we may assume that $D \neq 0$.
Let $A = \ker D$, $S = A \setminus \{0\}$, and $K=S^{-1}A$.
Then $S^{-1}D \in \lnd( S^{-1}B )$ and there exists $t \in S^{-1}B$ such that $(S^{-1}D)(t)=1$.
By \ref{pc9293ed0wdjo03}, $S^{-1}B = K[t] = K^{[1]}$ and $S^{-1}D = \frac{d}{dt}$.
Now $\exp(D) \in \Aut(B)$ extends to $\exp( \frac{d}{dt} ) \in \Aut( K[t] )$, which is the
$K$-automorphism of $K[t]$ that sends $t$ to $t+1$. Consequently, we have $\Frac B = K(t)$ and $E_D : K(t) \to K(t)$ is the
$K$-automorphism that sends $t$ to $t+1$.
We leave it as an exercise to check that $\setspec{ \xi \in K(t) }{ E_D(\xi) = \xi }$ is equal to $K$.
This proves the Lemma.
\end{proof}

Next, we consider how $\Kscr(B)$ behaves under extension of the base field.

\begin{parag} \label {Wjd9f23u8n42ndhje8}
Let $B$ be an algebra over a field $\bk$ of characteristic zero and let $D \in \lnd(B)$. 
Let $\ck$ be any field extension of $\bk$ and define $\bar B = \ck \otimes_\bk B$. 
Applying the functor $\ck \otimes_\bk (\underline{\ \ }) : \text{\rm $\bk$-Mod} \to \text{\rm $\ck$-Mod}$ to $D$
gives a $\ck$-linear map $\bar D : \bar B \to \bar B$, given by $\bar D( \lambda \otimes b ) = \lambda \otimes D(b)$
for all $\lambda \in \ck$ and $b \in B$. It is easily verified that $\bar D \in \lnd( \bar B )$, so we have a well-defined set map
$D \mapsto \bar D$ from $\lnd(B)$ to $\lnd(\bar B)$.
If $D \in \lnd(B)$ and $A = \ker D$, then $\ker(\bar D) = \ck \otimes_\bk A$ because
$\ck \otimes_\bk (\underline{\ \ })$ is an exact functor.
\end{parag}

\begin{lemma} \label {evdhd838eh65433ru0}
Let $\ck/\bk$ be an algebraic extension of fields of characteristic zero.
Let $B$ be an affine $\bk$-domain satisfying $\FML(B)=\bk$,
and define $\bar B = \ck \otimes_\bk B$.
\begin{enumerata}

\item  $\bar B$ is an affine $\ck$-domain satisfying $\FML(\bar B)=\ck$ and $\dim(\bar B) = \dim B$.

\item Each $D \in \lnd(B)$ has a unique extension $\bar D \in \lnd(\bar B)$.
Every subset $\Delta$ of $\lnd(B)$ determines a subset $\bar\Delta$ of $\lnd(\bar B)$ defined by $\bar \Delta = \setspec{ \bar D }{ D \in \Delta }$.
We have
$$
\ck \otimes_\bk K_\Delta = K_{\bar\Delta} \in \Kscr(\bar B) \quad \text{for every subset $\Delta$ of $\lnd(B)$.}
$$
In particular, for each $K \in \Kscr(B)$ we have $\ck \otimes_\bk K \in \Kscr(\bar B)$. 

\item The map
$$
\begin{array}{rcl}
\Kscr(B)    &    \to    &    \Kscr( \bar B )   \\[1mm]
K  &  \mapsto & \ck \otimes_\bk K
\end{array}
$$
is injective and preserves transcendence degree:
$$
\trdeg( \Frac(B) : K ) = \trdeg( \Frac(\bar B) : \ck \otimes_\bk K ) \quad \text{for all $K \in \Kscr(B)$.}
$$

\end{enumerata}
\end{lemma}

\section{Rings having trivial FML-invariant}
\label {RingshavingtrivialFMLinvariant}

\begin{definition}
Let $\ck$ be the algebraic closure of a field $\bk$.
An affine $\bk$-domain $B$ is said to be \textit{geometrically rational} (resp.\ \textit{geometrically unirational}) over $\bk$
if $\ck \otimes_\bk B$ is a domain and the field extension $\Frac(\ck \otimes_\bk B ) / \ck$ is rational (resp.\ unirational).
\end{definition}

The following is a straightforward consequence of the Unirationality Theorem (stated in the introduction) and of Lemma \ref{evdhd838eh65433ru0}:

\begin{corollary} \label {9132g8rf7rhf}
Let $\bk$ be a field of characteristic zero and $B$ an affine $\bk$-domain satisfying \mbox{$\FML(B) = \bk$.}
Then $B$ is geometrically unirational over $\bk$.
\end{corollary}

\begin{proof}
Consider the algebraic closure $\ck$ of $\bk$.
By Lemma \ref{evdhd838eh65433ru0}, $\bar B = \ck \otimes_\bk B$ is an affine $\ck$-domain satisfying $\FML(\bar B) = \ck$.
By the Unirationality Theorem, it follows that $\bar B$ is unirational over $\ck$.
\end{proof}

Our objective for the rest of this section
is to show that if $\FML(B)=\bk$ then $B$ satisfies a condition stronger than geometric unirationality.
This is achieved in Thm \ref{9i3oerXfvdf93p04efeJ} (see also Rem.\ \ref{p0c9v023we9f0d}).

\begin{notations} \label {iefocXo23wiedn}
Given an algebra $B$ over a field $\bk$ of characteristic zero
and a finite sequence $\Seul = (D_1, \dots, D_N)$  of elements of $\lnd(B)$
(where $N\ge1$ and where $D_1,\dots,D_N$ are not necessarily distinct),
we proceed to define a $\bk$-homomorphism $\Psi_\Seul : B \to B[X]$, where $B[X]$ is the polynomial ring $B[X_1,\dots,X_N]= B^{[N]}$.

For each $i \in \{1,\dots,N\}$, let $\delta_i \in \lnd( B[X_1,\dots,X_N] )$ be the unique extension of $D_i$ such that $\delta_i(X_j)=0$ for all $j \in \{1, \dots, N\}$,
note that $X_i\delta_i \in \lnd( B[X_1,\dots,X_N] )$,
and let $\epsilon_i = \exp( X_i \delta_i ) \in \Aut_\bk( B[X_1,\dots,X_N] )$.
Define $\Psi_\Seul : B \to B[X_1,\dots,X_N] $ to be the composition
$$
B \hookrightarrow B[X_1,\dots,X_N] \xrightarrow{ \epsilon_N \circ \cdots \circ \epsilon_1 } B[X_1,\dots,X_N] 
$$
and note that $\Psi_\Seul$ is explicitly given by
\begin{equation}  \label {xncb2oi3qwdowes}
\Psi_\Seul(b) =
\sum_{(i_1, \dots, i_N) \in \Nat^N}
\left( \textstyle \frac{(D_N^{i_N} \circ \cdots \circ D_1^{i_1})(b)}{i_1 ! \cdots i_N !}  \right) X_1^{i_1} \cdots X_N^{i_N} 
\quad \text{for all $b \in B$.}
\end{equation}

By the special case $R=B$ of Notations \ref{90q3985yTghvj29hvnr},
$\Psi_\Seul$ determines for each $\pgoth \in \Spec B$ 
\begin{itemize}

\item a $\bk$-homomorphism $\Psi_\Seul^\pgoth : B \to \kappa(\pgoth)[X]$
\item and a $\kappa(\pgoth)$-homomorphism $\hat\Psi_\Seul^\pgoth : \kappa(\pgoth) \otimes_\bk B \to \kappa(\pgoth)[X]$,
\end{itemize}
where $\kappa(\pgoth) = B_\pgoth/\pgoth B_\pgoth$ and $\kappa(\pgoth)[X] =  \kappa(\pgoth)[X_1,\dots,X_N]= \kappa(\pgoth)^{[N]}$.
Let us recall how these homomorphisms are defined.
Let $\phi_\pgoth : B \to \kappa(\pgoth)$ be the canonical homomorphism
and let $\tilde\phi_\pgoth : B[X] \to \kappa(\pgoth)[X]$ be the unique extension of $\phi_\pgoth$
satisfying $\tilde\phi_\pgoth(X_i)=X_i$ for all $i$. 
Then let $\Psi_\Seul^\pgoth$ be the composition
$$
B \xrightarrow{ \Psi_\Seul } B[X_1, \dots, X_N]  \xrightarrow{ \tilde\phi_\pgoth } \kappa(\pgoth)[X_1, \dots, X_N] 
$$
and note that
\begin{equation}  \label {o3874t5ftc3945h7r8f}
\Psi_\Seul^\pgoth (b) =
\sum_{(i_1, \dots, i_N) \in \Nat^N}
\phi_\pgoth\left( \textstyle \frac{(D_N^{i_N} \circ \cdots \circ D_1^{i_1})(b)}{i_1 ! \cdots i_N !}  \right) X_1^{i_1} \cdots X_N^{i_N} 
\quad \text{for all $b \in B$.}
\end{equation}
The $\kappa(\pgoth)$-homomorphism $\hat\Psi_\Seul^\pgoth$ is defined via the pushout square \eqref{892bf9jHGTiguYTFUKHiu73}; it satisfies
$$
\begin{array}{rcl}
\hat\Psi_\Seul^\pgoth \, : \, \kappa(\pgoth) \otimes_\bk B & \longrightarrow  & \kappa(\pgoth)[X_1, \dots, X_N] . \\
\lambda \otimes b \   & \longmapsto & \ \lambda \Psi_\Seul^\pgoth(b)
\end{array}
$$
\end{notations}

\begin{notations} 
\begin{enumerate}
\item If $\Psi : R \to S$ is a ring homomorphism, we write $\Psi^* : \Spec S \to \Spec R$ for the morphism of schemes determined by $\Psi$.
\end{enumerate}
Let $\bk$ be an algebraically closed field of characteristic zero, $B$ an affine $\bk$-domain, and $X = \Spec B$.
\begin{enumerate}
\addtocounter{enumi}{1}

\item If $\Seul = (D_1, \dots, D_N)$ is a finite sequence of elements of $\lnd(B)$,
we consider the morphism of $\bk$-varieties $\Psi_\Seul^* : \aff^N \times X \to X$ determined by the $\bk$-homomorphism 
$\Psi_\Seul : B \to B[X_1,\dots,X_N] $ defined in \ref{iefocXo23wiedn}.

\item
Let $\Delta$ be a subset of $\lnd(B)$.
For each  $\lambda \in \bk$ and $D \in \Delta$, we have $\exp(\lambda D) \in \Aut_\bk(B)$ and hence $\exp(\lambda D)^* \in \Aut(X)$.
Let $G_\Delta$ be the subgroup of $\Aut(X)$
generated by the set $\setspec{ \exp(\lambda D)^* }{ \lambda \in \bk,\ D \in \Delta }$.
For each closed point $x \in X$, let $G_\Delta(x) \subseteq X$ denote the orbit of $x$ with respect to the natural action of $G_\Delta$ on $X$.

\end{enumerate}
\end{notations}

\begin{lemma}  \label {fh939fg2730fikqzzzzz}
Let $\bk$ be an algebraically closed field of characteristic zero, $B$ an affine $\bk$-domain,
and $\Delta$ a subset of $\lnd(B)$.  Let $X = \Spec B$.
\begin{enumerata}

\item There exists a finite sequence $\Seul = (D_1, \dots, D_N)$ of elements of $\Delta$ such that the morphism
$\Psi_\Seul^* : \aff^N \times X \to X$ has the following property:
\begin{center}
For every closed point $x \in X$,\ \  $\Psi_\Seul^* ( \aff^N \times \{x\} ) = G_\Delta(x)$.
\end{center}

\item There exists a nonempty Zariski-open subset $U$ of $X$ such that, for every closed point $x \in U$,
the dimension of $G_\Delta(x)$ is equal to the transcendence degree of $\Frac B$ over $K_\Delta$.

\end{enumerata}
\end{lemma}

This result is a corollary of Thm 1, Thm 2 and Cor.\ 2 of \cite{Popov_InfiniteDimAlgGps2014}.
The proof below explains how to use Popov's results to obtain Lemma \ref{fh939fg2730fikqzzzzz}.
One should read the definitions given on pages 551--552 of \cite{Popov_InfiniteDimAlgGps2014} before reading
this proof.

\begin{proof}[Proof of Lemma \ref{fh939fg2730fikqzzzzz}]
If $D \in \Delta$ and $\lambda \in \bk$ then $\exp(\lambda D) \in \Aut_\bk(B)$ and hence $\exp(\lambda D)^* \in \Aut(X)$.
Let us write $F_D = \{ \exp(\lambda D)^* \}_{ \lambda \in \bk } = \{ \exp(\lambda D)^* \}_{ \lambda \in \aff^1 }$,
then $F_D$ is a unital algebraic family in $\Aut(X)$ (this is defined on page 551 of  \cite{Popov_InfiniteDimAlgGps2014}). 
Thus, $\mathcal{I}_\Delta = \setspec{ F_D }{ D \in \Delta }$ is a collection of 
unital algebraic families in $\Aut(X)$ and $G_\Delta$ is generated by $\mathcal{I}_\Delta$.
It then follows from Lemma 1 of \cite{Popov_InfiniteDimAlgGps2014} that $G_\Delta$ is a connected subgroup of $\Aut(X)$
(this notion is defined in the second paragraph of page 552 of \cite{Popov_InfiniteDimAlgGps2014}). 

Note that if $\Seul = (D_1,\dots,D_N)$ is a finite sequence of elements of $\Delta$ then the morphism
$\Psi_\Seul^* : \aff^N \times X \to X$ is such that $\Psi_\Seul^*(t,\underline{\ \ }) \in \Aut(X)$ for each $t \in \aff^N$,
so if we define $F_\Seul = \{ \Psi_\Seul^*(t,\underline{\ \ }) \}_{t \in \aff^N}$ then 
$F_\Seul$ is a unital algebraic family in $\Aut(X)$.
Moreover, a moment's reflexion shows that
$ \setspec{ F_\Seul }{ \text{$\Seul$ is a finite sequence of elements of $\Delta$} } $
is precisely the set of all families in $\Aut(X)$ that are derived from $\mathcal{I}_\Delta$ (this concept is defined in
the last paragraph of page 551 of  \cite{Popov_InfiniteDimAlgGps2014}).

By Thm 1 of \cite{Popov_InfiniteDimAlgGps2014},
there exists a family derived from $\mathcal{I}_\Delta$ and exhaustive for the natural action of $G_\Delta$ on $X$
(defined on page 552). This proves assertion (a) of the Lemma.

Note that if we identify the function field $\bk(X)$ of $X$ with $\Frac(B)$,
then Lemma \ref{73npgjk32ncrg93rcn} implies that $\bk(X)^{ G_\Delta } = K_\Delta$.
Thus 
\begin{equation} \label {9r8f67382efe}
\trdeg\big( \Frac(B) : K_\Delta \big) = \trdeg \big( \bk(X) : \bk(X)^{G_\Delta} \big).
\end{equation}
So we may apply Thm 2 and Cor.\ 2 of \cite{Popov_InfiniteDimAlgGps2014} as follows:
\begin{itemize}

\item[(i)] By Thm 2, there exist an $m \in \Nat$ and a nonempty open subset $U$ of $X$ such that, for every closed point $x \in U$,
the dimension of $G_\Delta(x)$ is equal to $m$.

\item[(ii)] By Cor.\ 2,
the transcendence degree of $\bk(X)^{G_\Delta}$ over $\bk$ is equal to $\dim X - m$, where $m$ is the same as in (i).

\end{itemize}
This gives $m = \dim X - \trdeg\big( \bk(X)^{G_\Delta} : \bk \big)
= \trdeg \big( \bk(X) : \bk(X)^{G_\Delta} \big)$,
which is equal to $\trdeg\big( \Frac(B) : K_\Delta \big)$ by \eqref{9r8f67382efe}.
This proves assertion (b) of the Lemma.
\end{proof}

The following is an intermediate result that will be improved in Thm \ref{9i3oerXfvdf93p04efeJ}:

\begin{lemma} \label {xjcp293iqwpd9icwso}
Let $\bk$ be a field of characteristic zero and $B$ an affine $\bk$-domain satisfying $\FML(B) = \bk$.
Let $\Delta$ be any subset of $\lnd(B)$ satisfying $K_\Delta = \bk$.
Then there exist a finite sequence $\Seul = (D_1, \dots, D_N)$  of elements of $\Delta$ and a nonempty Zariski-open subset $U$ of $\Spec B$ such that,
for each maximal ideal $\mgoth$ of $B$ belonging to $U$, the $\bk$-homomorphism $\Psi_\Seul^\mgoth : B \to \kappa(\mgoth)[X_1, \dots, X_N]$ is injective.
\end{lemma}

\begin{proof}
Let us first prove the Lemma under the additional assumption that $\bk$ is algebraically closed.
Consider any subset $\Delta$ of $\lnd(B)$ satisfying $K_\Delta = \bk$.
By Lemma \ref{fh939fg2730fikqzzzzz},
there exists a finite sequence $\Seul = (D_1, \dots, D_N)$ of elements of $\Delta$ such that the morphism
$\Psi_\Seul^* : \aff^N \times X \to X$ (where $X=\Spec B$) satisfies
$\Psi_\Seul^* ( \aff^N \times \{x\} ) = G_\Delta(x)$ for every closed point $x \in X$.
Choose such a sequence $\Seul$. Lemma \ref{fh939fg2730fikqzzzzz} also implies that
there exists a nonempty open subset $U$ of $\Spec B$ such that, for every closed point $x \in U$,
the dimension of $G_\Delta(x)$ is equal to the transcendence degree of $\Frac B$ over $K_\Delta = \bk$.
Choose such a $U$ and note that if $x$ is a closed point of $U$ then $G_\Delta(x)$ is dense in $\Spec B$.
This means that if $x$ is a closed point of $U$ then the composition
$$
\begin{array}{ccccc}
\aff^N & \to & \aff^N \times X & \xrightarrow{ \Psi_\Seul^* } & X \\
t & \mapsto & (t,x)
\end{array}
$$
is a dominant morphism.  Now this composition is precisely $(\Psi_\Seul^\mgoth)^*$, where $\mgoth = x$, 
so $\Psi_\Seul^\mgoth$ is injective (and this is true for every maximal ideal $\mgoth$ of $B$ such that $\mgoth \in U$).
This proves the Lemma when $\bk$ is algebraically closed.

Now drop the assumption that $\bk$ is algebraically closed.
Consider any subset $\Delta$ of  $\lnd(B)$ satisfying $K_\Delta = \bk$.
Let $\ck$ be the algebraic closure of $\bk$ and let $\bar B = \ck \otimes_\bk B$.
By Lemma \ref{evdhd838eh65433ru0}, $\bar B$ is an affine $\ck$-domain satisfying $\FML(\bar B) = \ck$.
Define $\bar\Delta \subseteq \lnd(\bar B)$ as in Lemma \ref{evdhd838eh65433ru0}, then (by Lemma \ref{evdhd838eh65433ru0})
$K_{\bar\Delta} = \ck \otimes_\bk K_\Delta = \ck \otimes_\bk \bk$, so $K_{\bar\Delta} = \ck$.
Since the Theorem is valid for $\bar B$ by the first part of the proof, 
there exist a finite sequence $\bar\Seul = (\bar D_1, \dots, \bar D_N)$  of elements of $\bar\Delta$ and a dense open subset $\bar U$ of $\Spec \bar B$ such that,
for each maximal ideal $\bar \mgoth$ of $\bar B$ belonging to $\bar U$,
$\Psi_{\bar\Seul}^{\bar\mgoth} : \bar B \to \kappa(\bar\mgoth)[X_1, \dots, X_N]$ is injective.
Let $\Seul = (D_1,\dots,D_N)$ be the sequence of elements of $\Delta$ such that $\bar D_i$ is the extension of $D_i$ for all $i$.
Let $U$ be the image of $\bar U$ by the open map $\Spec \bar B \to \Spec B$
(by \cite[Tag 037G]{stacks-project},
if $R$ and $S$ are algebras over a field $k$ then $\Spec(R \otimes_k S) \to \Spec R$ is an open morphism).
Then $U$ is a dense open subset of $\Spec B$ and we claim that $\Seul$ and $U$ have the desired property.
Indeed, let $\mgoth$ be any maximal ideal of $B$ satisfying $\mgoth \in U$.
Then there exists $\bar\mgoth \in \bar U$ satisfying $\bar\mgoth \cap B = \mgoth$,
and $\bar\mgoth$ is in fact a maximal ideal of $\bar B$ because $\bar B$ is integral over $B$.
It is clear from equations \eqref{xncb2oi3qwdowes} and \eqref{o3874t5ftc3945h7r8f} that the diagram 
$$
\xymatrix{
\bar B \ar[r]^-{\Psi_{\bar\Seul}} & \bar B[X_1,\dots,X_N] \ar[r]^-{\Phi_{\bar\mgoth}} &  \kappa(\bar\mgoth)[X_1, \dots, X_N] \\
B \ar[r]_-{\Psi_\Seul} \ar[u] & B[X_1,\dots,X_N] \ar[r]_-{\Phi_\mgoth} \ar[u] &  \kappa(\mgoth)[X_1, \dots, X_N] \ar[u]
}
$$
commutes.  Since $B \to \bar B$ and  $\Psi_{\bar\Seul}^{\bar\mgoth} = \Phi_{\bar\mgoth} \circ \Psi_{\bar\Seul}$  are injective,
it follows that $\Psi_{\Seul}^{\mgoth} = \Phi_{\mgoth}\circ\Psi_{\Seul}$  is injective.
This completes the proof of Lemma \ref{xjcp293iqwpd9icwso}.
\end{proof}

The following well-known fact is very useful: 

\begin{lemma} \label {0c9v902j3we0fcwpdjhf2873} 
Let $\bk$ be a field of characteristic zero and $B$ a $\bk$-domain.
The following are equivalent:
\begin{enumerata}

\item $\bk$ is algebraically closed in $\Frac(B)$
\item $K \otimes_\bk B$ is a domain for every extension field $K$ of $\bk$ (i.e., $B$ is geometrically integral).
\item $\ck \otimes_\bk B$ is a domain, where $\ck$ is the algebraic closure of $\bk$.

\end{enumerata}
\end{lemma}


See \ref{p0c9f2n3w9sdc0qpwo} for the notation $\Xeul_\bk(B)$.

\begin{theorem} \label {9i3oerXfvdf93p04efeJ}
Let $\bk$ be a field of characteristic zero and $B$ an affine $\bk$-domain. If $\FML(B) = \bk$ then the following hold.
\begin{enumerata}

\item The set $\Xeul_\bk(B)$ has nonempty interior, i.e., the condition $\kappa(\pgoth) \otimes_\bk B \subseteq \kappa(\pgoth)^{[ \dim B ]}$
is satisfied for all prime ideals $\pgoth$ in some nonempty open subset of $\Spec B$.

\item Given any subset $\Delta$ of $\lnd(B)$ satisfying $K_\Delta = \bk$,
there exist a finite sequence $\Seul = (D_1, \dots, D_N)$  of elements of $\Delta$ and a nonempty Zariski-open subset $U$ of $\Spec B$ such that,
for each $\pgoth \in U$,
the $\kappa(\pgoth)$-homomorphism $\hat\Psi_\Seul^\pgoth : \kappa(\pgoth) \otimes_\bk B \longrightarrow  \kappa(\pgoth)[X_1, \dots, X_N]$ is injective.

\end{enumerata}
\end{theorem}

\begin{proof} 
The condition  $\FML(B) = \bk$ implies that $\bk$ is algebraically closed in $\Frac B$ by Lemma \ref{0ckj238hqn239hnfpawrnHzb},
so $B$ is geometrically integral by Lemma \ref{0c9v902j3we0fcwpdjhf2873}.
We first prove (b).
Let $\Delta$ be a subset of $\lnd(B)$ satisfying $K_\Delta = \bk$.
By Lemma \ref{xjcp293iqwpd9icwso},
there exist a finite sequence $\Seul = (D_1, \dots, D_N)$  of elements of $\Delta$ and a nonempty Zariski-open subset $U$ of $\Spec B$ such that,
for each maximal ideal $\mgoth$ of $B$ belonging to $U$, the $\bk$-homomorphism $\Psi_\Seul^\mgoth : B \to \kappa(\mgoth)[X_1, \dots, X_N]$ is injective.
Since $B$ is geometrically integral,
part \eqref{9hd8HGuKHytYf7624i-iff} of Thm \ref{cp0Q9vn23we9dfcwend0} implies that
$\mgoth \in \Xeul_\bk( \Psi_\Seul )$ for each maximal ideal $\mgoth$ of $B$ belonging to $U$,
so in particular $\Xeul_\bk( \Psi_\Seul )$ is dense in $\Spec B$.
Since $\Xeul_\bk( \Psi_\Seul )$ is a constructible subset of $\Spec B$ by part \eqref{9hd8HGuKHytYf7624i-cons} of Thm \ref{cp0Q9vn23we9dfcwend0},
some\footnote{Actually, one can show that $U$ itself is included in $\Xeul_\bk( \Psi_\Seul )$.}
nonempty open subset of $\Spec B$ is included in $\Xeul_\bk( \Psi_\Seul )$, so (b) follows.
Since  $\Xeul_\bk( \Psi_\Seul ) \subseteq \Xeul_\bk(B)$, (a) follows.
\end{proof}

\begin{bigremarks} \label {p0c9v023we9f0d}
Given a field $\bk$ and an affine $\bk$-domain $B$, consider the conditions
\begin{enumerate}

\item[(i)] $\Xeul_\bk(B)$ has nonempty interior;

\item[(ii)] $\ck \otimes_\bk B \subseteq \ck^{[ \dim B ]}$, where $\ck$ is the algebraic closure of $\bk$;

\item[(iii)] $B$ is geometrically unirational.

\end{enumerate}
Then (i) is  strictly stronger than (ii) by Cor.\ \ref{kxjcop923wpeds} and Ex.\ \ref{pc09vnE230ed9vCf}, 
and it is clear that (ii) is strictly stronger than (iii).
This justifies our claim that Thm \ref{9i3oerXfvdf93p04efeJ} goes beyond Cor.\ \ref{9132g8rf7rhf}.
\end{bigremarks}

\begin{corollary} \label {0c9f982u3e9fjpwe}
Let $\bk$ be a field of characteristic zero and let $B$ be an affine $\bk$-domain satisfying $\FML(B) = \bk$.
Then the following are equivalent:
\begin{enumerata}
\item  $B \subseteq \bk^{[ \dim B ]}$ 
\item  $\bk$-rational points are dense in $\Spec B$
\item  the extension $\Frac(B) / \bk$ has the density property
\item  $B$ is unirational over $\bk$.
\end{enumerata}
In particular, if $\bk$ is algebraically closed or $\dim B \le 2$ then $B \subseteq \bk^{[ \dim B ]}$.
\end{corollary}

\begin{proof}
Conditions (a--d) are equivalent by Rem.\ \ref{d023kHhvjhhJHkru83ee}.
If $\bk$ is algebraically closed then (b) holds, so $B \subseteq \kk{\dim B}$.
If $\dim B \le 2$ then it is well known that $B$ is rational over $\bk$
(see for instance \cite[Lemma 5.3.8]{Kol:thesis} or \cite[Cor.\ 2.6]{Daigle:StructureRings}),
 so (d) holds and hence $B \subseteq \kk{\dim B}$.
\end{proof}

\begin{question} \label {ifhbcoser328}
Let $\bk$ be a field of characteristic zero and $B$ an affine $\bk$-domain.
When $\bk$ is not algebraically closed, does the condition $\FML(B)=\bk$ imply that $B$ is unirational over $\bk$?
(Probably not, but an example would be welcome.)
\end{question}

\bibliographystyle{alpha}
\newcommand{\etalchar}[1]{$^{#1}$}

\end{document}